\documentclass[a4paper,12pt]{article}

\usepackage{tgtermes}

\usepackage{amsfonts,anysize,amsmath,amsthm,geometry}
\usepackage{bm}
\marginsize{30mm}{25mm}{18mm}{18mm}
\usepackage{lineno}
\usepackage[colorlinks,citecolor=red,linkcolor=blue]{hyperref}
\usepackage{mathrsfs}
\usepackage{amssymb}
\usepackage{stmaryrd}
\usepackage{array}
\usepackage{graphicx}
\usepackage{booktabs}
\usepackage{multirow}
\usepackage{algorithm}
\usepackage[figuresright]{rotating}
\usepackage{epstopdf}
\usepackage{algorithmic}

\numberwithin{equation}{section}
\newtheorem{theorem}{Theorem}[section]
\newtheorem{lemma}[theorem]{Lemma}

\theoremstyle{definition}

\newtheorem{remark}{Remark}[section]
\newtheorem{example}{Example}[section]

\def \rank {\mathrm{rank}}

\begin{document}
%\linenumbers

\title{An efficient numerical method for condition number constrained covariance matrix approximation \thanks{The work was supported by Shandong Provincial Natural Science Foundation (Grant No. ZR2020QA034)  and the author was also partly supported by the Shandong Provincial Natural Science Foundation (Grant No. ZR2019MA002).}}

\author{Shaoxin Wang\footnote{Corresponding author. School of Statistics, Qufu Normal University,  Qufu, 273165, P. R. China. Email addresses: shxwang@qfnu.edu.cn  and  shwangmy@163.com.}}

\date{\small Last modified on \today}

\maketitle

\begin{abstract}
  In the high-dimensional data setting, the sample covariance matrix is singular. In order to get a numerically stable and positive definite modification of the sample covariance matrix in the high-dimensional data setting, in this paper we consider the condition number constrained covariance matrix approximation problem and present its explicit solution with respect to the Frobenius norm. The condition number constraint guarantees the numerical stability and positive definiteness of the approximation form simultaneously. By exploiting the special structure of the data matrix in the high-dimensional data setting, we also propose some new algorithms based on efficient matrix decomposition techniques. Numerical experiments are also given to show the computational efficiency of the proposed algorithms.
\end{abstract}



%\newpage
\section{Introduction}
\label{sect.1}
In multivariate statistical analysis, the covariance matrix is a fundamental component of many statistical models, such as linear (quadratic) discriminant analysis, principal component analysis and canonical correlation analysis \cite{Ande03}. The covariance matrix also finds  its popularity in many other applied disciplines. Examples include classification of gene expression \cite{KhanW01}, machine learning \cite{CaoB09}, portfolio management \cite{Mkwz52,Ldot04}and so on. However, in most of practical applications, the underlying true covariance matrix is never to be known and usually estimated from the given data. Thus many covariance matrix estimation procedures have been developed under various data settings, and the interested readers should be referred to \cite{CaoB09,Ldot04,XueM12,LiuW14,WonL13} and the references therein.

The Big Data era brings much more complex data settings and challenges for covariance matrix approximation. Especially, the high-dimensional data analysis, where the dimension of variables $p$ is much larger than the number of observations $n$, is a very active topic of scientific research in the big data analysis \cite{KhanW01,CaoB09,Ldot04}. In such data setting, although the following sample covariance matrix
\begin{equation}\label{eq.defSn}
  {S}_n=\frac{1}{n}\sum_{i=1}^n{x}_i{x}_i^T=\frac{1}{n}X_nX_n^T \quad\text{with}\; X_n=[x_1,\cdots,x_n]\text{ and } x_i\in \mathbb{R}^{p}
\end{equation}
maximizes the likelihood function of centered data under a normal model and is an asymptotically unbiased estimate of the underlying true covariance matrix $\Sigma$, it may have poor performance in approximating the eigenstructure of $\Sigma$, especially when $p$ is close to or larger than $n$ \cite{Ldot04,Bai08}.
To show this, let ${x}_1, \cdots, {x}_n$ be observations of a random vector $\mathbf{x} \in \mathbb{R}^{p}$ from a multivariate normal distribution $\mathcal{N}_p(0,\Sigma)$. We generate 100 groups of data from $\mathcal{N}(0,I_p)$, and plot the mean of the $i$-th largest eigenvalues of $S_n$ ($i=1, \cdots, p$) in Figure~\ref{toyexample}.
\begin{figure}[htp]
  \centering
  \includegraphics[width=1\textwidth,height=0.34\textwidth]{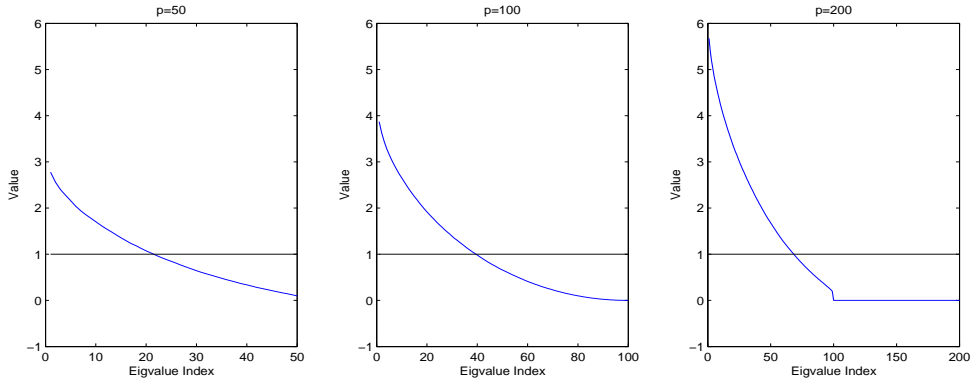}\\
  \caption{The distribution of eigenvalues of ${S}_n={1}/{n}\sum_{i=1}^n{x}_i{x}_i^T$ with $x_i\sim \mathcal{N}(0,I_p)$ and $n=100$.}\label{toyexample}
\end{figure}
{\color {black} Since we set $\Sigma=I_p$, the true eigenvalues of $\Sigma$ are equal to 1 and $\Sigma$ is also well-conditioned. However, from Figure~\ref{toyexample} we can find that when $n$ approaches $p$ the eigenvalues of $S_n$ become more and more dispersed. When $p>n$, $S_n$ may not only have dispersed eigenvalues but also become singular. Thus, in order to get a suitable approximation of $\Sigma$ from such data setting, a natural idea is to remove or decrease the excess dispersion of the extreme eigenvalues and make the final approximation form positive definite.}

To approximate $\Sigma$ in the high-dimensional data setting, numerous statistical or numerical methods have been proposed with respect to different criterions (cf.\cite{FanLL16,High18}). In statistics, the main research interest on high-dimensional covariance matrix estimation is to establish the statistical consistency property of the estimators, which relies on the assumptions of sparse covariance structure and the moment conditions of random variables \cite{BicL08,BienT11,LaFa09,RotLZ09}. However, in practical computation,  the computationally efficient thresholding techniques may lead the estimates to be indefinite \cite{BicL08,RotLZ09}. Therefore, some shrinkage and constraint methods are proposed to ensure the positive definiteness of the estimates \cite{XueM12,LiuW14,Ldot04,Ldot12,High18}. Considering numerical stability, Won et al. \cite{WonL13} proposed the condition number constrained maximum likelihood estimator of covariance matrix, which enjoys the statistical consistency property under some strong assumptions of covariance structure \cite{GuoZ17}.  The idea was extended to the maximum likelihood estimation of structured covariance matrix \cite{AubD12} and precision matrix \cite{GuoZ17}.  The introduction of condition number constraint can not only guarantee the positive definiteness of estimated covariance matrix but also its numerical stability.

The assumptions on covariance structure and the distribution of random variables may be restrictive and hard to verify in practical applications. Under different loss functions or transformations, some optimal matrix approximation techniques are proposed to obtain positive semidefinite covariance matrix approximations \cite{High02,Duan16,Duan14,FanLL16}. Considering positive definiteness and numerical stability, Tanaka and Nakata \cite{TanaN14} considered condition number constrained positive definite matrix approximation problem with unitarily invariant norm, which can be used to get the covariance matrix estimation from a modified sample covariance matrix. Although the authors suggested that the spectral decomposition and binary search can be used to find the optimal solution, calculating all the eigenvalues and eigenvectors may be infeasible in the high-dimensional data setting. In this paper, we reconsider the condition number constrained covariance matrix approximation problem with a different norm and present a new characterization of the solution. By exploring the structure of data matrix in high-dimension data setting, we will also investigate efficient numerical algorithms for finding the optimal solution.

The reminder of the paper is organized as follows. In Section~\ref{sec.2}, a reformulation of condition number constrained covariance matrix approximation problem will be given, and we also give a comparison with the existing works. Our main results are given in Section~\ref{sec.3}, in which we present the explicit expression of the solution to condition number constrained covariance matrix approximation problem and propose some new and efficient algorithms. To show the efficiency of our results, the numerical experiments are performed in Section~\ref{sec.4}. Section~\ref{sec.5} contains the concluding remark of the whole paper.

\section{Problem reformulation}
\label{sec.2}
The main idea to overcome the excess dispersion of sample eigenvalues is to pull the extreme sample eigenvalues back to some target, so some shrinkage methods are proposed \cite{Ldot04,Ldot12,High18}. However, we note that when $p$ is much larger than $n$, the linear shrinkage techniques may give unreliable approximation of covariance matrix \cite{WangY19,Ldot04,Ldot12}. To remove the influence of extreme sample eigenvalues, a more straightforward method is to directly bound the eigenvalues or condition number of covariance matrix, which has been widely used in high-dimensional covariance matrix estimation with respect to different loss functions \cite{XueM12, WonL13,LiuW14,GuoZ17}.

In this work,
we also confine ourself to the high-dimensional data setting, but we will reconsider the condition number constrained covariance matrix approximation (C3MA) problem under the assumption that the data matrix $X_n$ is of full-column rank. The C3MA problem can be stated as follows
\begin{equation}
\label{eq.covcd}
\hat{\Sigma}=\mathop\mathrm{argmin}_{\Sigma \in \mathcal{P}_{+}^{p}} \left\| \Sigma- S_n \right\|_F, \quad\text{ subject to }\;\kappa(\Sigma)\leq \kappa_n,
\end{equation}
where $S_n$ is defined by \eqref{eq.defSn},  $\left\| \cdot \right\|_F$ denotes the Frobenius norm, $\mathcal{P}_{+}^{p}$  is the cone of $p\times p$ positive semidefinite matrices, $\kappa(\Sigma)=\lambda_{1}(\Sigma)/\lambda_{p}(\Sigma)$ is the condition number of $\Sigma$ with $\lambda_{1}$ and $\lambda_{p}$ being the maximum and minimum eigenvalues of $\Sigma$, and $\kappa_n\in[1,\infty)$ is a finite positive real number.
Similar to \cite{TanaN14}, the C3MA problem defined by \eqref{eq.covcd} is almost equivalent to the following problem
\begin{equation}
\label{eq.covcd2}
\hat{\Sigma}=\mathop\mathrm{argmin}_{\Sigma \in \mathcal{P}_{+}^{p}} \left\| \Sigma- S_n \right\|_F, \quad\text{ subject to }\;\lambda_1\leq \kappa_n\lambda_p.
\end{equation}
The only difference between \eqref{eq.covcd} and \eqref{eq.covcd2} is that $\Sigma = 0$ is the feasible solution to \eqref{eq.covcd2} but can not be the solution to \eqref{eq.covcd}, when $S_n=0$. In this work, we employ the assumption that  $X_n$ is of full-column rank to exclude this trivial case, and a detailed comparison with the existing works will be given in Section~\ref{subsec:2.1}.
According to \eqref{eq.covcd}, we can realize that $\hat{\Sigma}$ is not only positive definite but also enjoys some numerical stability. It can also be checked that if $S_n$ is positive definite and  $\kappa(S_n)\leq \kappa_n$ then $S_n$ is the optimal solution. However, in the high-dimensional data setting, $S_n$ is singular and thus can not be the feasible solution to \eqref{eq.covcd}.  Moreover, it is easy to check that the C3MA problem \eqref{eq.covcd} is convex and its proof can be similarly derived as \cite{TanaN14b}.

\subsection{Related work and useful result}
\label{subsec:2.1}

The idea of using condition number constraint to ensure the numerical stability of covariance matrix may be first used in maximum likelihood estimation of covariance matrix with applications in portfolio management and radar signal processing \cite{Won06,WonL13,AubD12}. Without considering the distribution of random variable, some scholars considered condition number constrained covariance matrix approximation problem. For example, Tanaka and Nakata considered the positive definite matrix approximation with condition number constraint \cite{TanaN14}. Then the result was extended to condition number constrained non-square matrix approximation problem arising from communication systems \cite{TonG14}. There are also some other related works on various matrix approximation problem with condition number constraint, and interested readers are referred to \cite{GuoZ17,TanaN14b,LLV20}.

The most relevant works are given by Tanaka and Nakata \cite{TanaN14}  and  Li et al. \cite{LLV20}.  Tanaka and Nakata \cite{TanaN14} studied the C3MA problem \eqref{eq.covcd} with the norm to be unitarily invariant and $S_n$ being only symmetric.  Li et al. \cite{LLV20} considered the same problem but with the norm being the Frobenius norm. The following Lemma~\ref{Lem.Tana1} was given in \cite{TanaN14} to characterize their solution. Since the Frobenius norm is also unitarily invariant and $S_n$ is symmetric by construction, Lemma~\ref{Lem.Tana1} can be directly applied to simplify the C3MA problem \eqref{eq.covcd}  and thus we present the corresponding result as Lemma~\ref{Lem.2} without proof.

\begin{lemma}
\label{Lem.Tana1}
Let the Frobenius norm $\|\cdot\|_F$ in \eqref{eq.covcd2} be replaced with any unitarily invariant norm $\|\cdot\|$ and ${S}_n$ only a symmetric matrix. Given the spectral decomposition ${S}_n=U\hat{\Lambda} U^T$ with $\hat{\Lambda}=\mathrm{Diag}(\hat{\lambda}_1,\cdots,\hat{\lambda}_p)$ and $\hat{\lambda}_1\geq\cdots\geq\hat{\lambda}_p$, if $\Sigma=0$ is not the feasible solution,  then for any optimal solution $\Lambda^{*}$ to the following problem
\begin{equation}
\label{eq.Tana1}
 \Lambda^{*}= \mathop\mathrm{argmin}_{\Lambda=\mathrm{Diag}(\lambda_1,\cdots,\lambda_p)}\left\|\Lambda- \hat{\Lambda}\right\|, \quad\text{ subject to }\;\min_i \lambda_i\geq 0 \;\; \text{and}\;\; \max_i \lambda_i\leq \kappa_n \min_i \lambda_i,
\end{equation}
$\hat{\Sigma}=U\Lambda^{*}U^T$ is the optimal solution to problem \eqref{eq.covcd2}.
\end{lemma}

\begin{lemma}
\label{Lem.2}
Assume that $X_n$ is of full-column rank, and the spectral decomposition of sample covariance matrix defined by \eqref{eq.defSn} is given by ${S}_n=U\hat{\Lambda} U^T$ with $\hat{\Lambda}=\mathrm{Diag}(\hat{\lambda}_1,\cdots,\hat{\lambda}_p)$ and $\hat{\lambda}_1\geq\cdots\geq\hat{\lambda}_n> \hat{\lambda}_{n+1}=\cdots=\hat{\lambda}_p=0$. Then for any optimal solution $\Lambda^{*}$ to the following problem
\begin{equation}
\label{eq.lem2}
 \Lambda^{*}= \mathop\mathrm{argmin}_{\Lambda=\mathrm{Diag}(\lambda_1,\cdots,\lambda_p)}\left\|\Lambda- \hat{\Lambda}\right\|_F, \quad\text{ subject to }\;\min_i \lambda_i\geq 0 \;\; \text{and}\;\; \max_i \lambda_i\leq \kappa_n \min_i \lambda_i,
\end{equation}
$\hat{\Sigma}=U\Lambda^{*}U^T$ is the optimal solution to the C3MA problem \eqref{eq.covcd}.
\end{lemma}

\begin{remark}
The assumption that $X_n$ is of full-column rank makes $\Sigma=0$ can not be the feasible solution to problem \eqref{eq.covcd2},  and  under this assumption these two problems \eqref{eq.covcd} and \eqref{eq.covcd2} are equivalent.
Though different methods have been proposed to reduce the condition number constrained covariance or precision matrix estimation problem into manipulation of its eigenvalues \cite{Won06,WonL13,GuoZ17,LLV20}, we can find that the main tools are Hoffman-Wielandt theorem \cite[pp.~368]{RogC94a} and Birkhoff's theorem \cite[pp.~527]{RogC94a}. With Lemma~\ref{Lem.Tana1}, Tanaka and Nakata showed that when the Ky Fan $p$-$k$ norm \cite[Ch.~3]{RogC94b} is used, the solution to \eqref{eq.Tana1} can be further simplified, and they also presented its analytical solutions with respect to spectral and trace norms \cite{TanaN14}. {\color{black}When the spectral decomposition of $S_n$ is available, Tanaka and Nakata  suggested a $O(p\log(p))$ binary search method to solve \eqref{eq.Tana1}. Motivated by the geometric idea given in \cite{Won06,WonL13}, we can show that the complexity can be further reduced, which was also considered in \cite{LLV20}. More importantly, apart from the searching procedure, we need to compute all the eigenvectors of $S_n$ to construct the final solution. Unfortunately, this can be infeasible in the high-dimensional data setting due to its heavy computational burden in computing the spectral decomposition of $S_n$}. However, the high-dimensional data setting imparts $S_n$ with a very spiked rank structure and it has at most $n\ll p$ nonzero eigenvalues. Thus it should be of interest to design some efficient numerical methods to solve \eqref{eq.covcd} by extensively exploiting the special structure of $S_n$.
\end{remark}

\section{Main results}
\label{sec.3}
\subsection{The solution of the C3MA problem}

According to our assumption on the C3MA problem \eqref{eq.covcd},  $X_n$ is of full-column rank and $p$ is much larger than $n$,  which ensures $S_n$ has very low rank and is singular. Thus $S_n$ can not be the feasible solution to \eqref{eq.covcd}. Taking the special structure of $S_n$ into consideration and with Lemma~\ref{Lem.2}, we present the solution of the C3MA problem \eqref{eq.covcd} in the following Theorem~\ref{Thm1} and its proof is given in Appendix.

\begin{theorem}
\label{Thm1}
Assume that $X_n$ is of full-column rank.
Let $S_n=U\hat{\Lambda} U^T$ be the spectral decomposition of $S_n$ with $\hat{\Lambda}=\mathrm{Diag}(\hat{\lambda}_1,\cdots,\hat{\lambda}_p)$ and $\hat{\lambda}_1\geq\cdots\geq\hat{\lambda}_n > \hat{\lambda}_{n+1}= \cdots\hat{\lambda}_{p}=0$. Then the solution to the C3MA problem \eqref{eq.covcd} is given by
\begin{eqnarray*}
  \hat{\Sigma}=U\Lambda^{*}U^T,
\end{eqnarray*}
where
\begin{eqnarray*}
  \Lambda^{*}=\mathrm{Diag}(\kappa_n\mu^{*},\cdots,\kappa_n\mu^{*}, \hat{\lambda}_{\alpha^{*}+1}, \cdots, \hat{\lambda}_{\beta^{*}-1}, \mu^{*},\cdots,\mu^{*}),
\end{eqnarray*}
\begin{eqnarray*}
  \mu^{*} = \frac{\mathop\kappa_{n}\sum_{i=1}^{\alpha^{*}}\hat{\lambda}_i+\sum_{i=\beta^{*}}^{n}\hat{\lambda}_i} {\alpha^{*}\kappa_{n}^2+p-\beta^{*}+1},
\end{eqnarray*}
and the optimal $\alpha^{*}$ and $\beta^{*}$ can be determined with about $O(n)$ operations.
\end{theorem}

\begin{remark}
\label{rmk1}
{\color {black}
Tanaka and Nakata only gave the analytical solution of \eqref{eq.covcd} with respect to spectral norm and trace norm \cite{TanaN14}. Considering the same problem, Li et al. presented its solution with respect to the Frobenius norm \cite{LLV20}. The main contribution of \cite{LLV20} with respect to \eqref{eq.covcd} is to show that when the Frobenius norm is used the optimal eigenvalues of $\hat{\Sigma}$ can be found in about $O(p)$ operations instead of $O(p\log(p))$. However, both of these two works gave little attention to calculating the eigenvectors of $\hat{\Sigma}$, and we will show that in the high-dimensional data setting this can be achieved with some efficient matrix decomposition techniques in Section~\ref{subsec.algthm}.

Moreover, when $n$ is larger than $p$, $S_n$ may be invertible.  Following the proof given in Appendix, we can show that when $\hat{\lambda}_1/\hat{\lambda}_p\leq \kappa_n$, $S_n$ is the optimal solution to \eqref{eq.covcd}. When $\hat{\lambda}_1/\hat{\lambda}_p > \kappa_n$, the optimal solution is given by
\begin{eqnarray}
\label{eq.rk1}
  \hat{\Sigma}=U\Lambda^{*}U^T,
\end{eqnarray}
where $\Lambda^{*}=\mathrm{Diag}(\kappa_n\mu^{*},\cdots,\kappa_n\mu^{*}, \hat{\lambda}_{\alpha^{*}+1}, \cdots, \hat{\lambda}_{\beta^{*}-1}, \mu^{*},\cdots,\mu^{*})$,
\begin{eqnarray*}
\mu^{*} = \frac{\mathop\kappa_{n}\sum_{i=1}^{\alpha^{*}}\hat{\lambda}_i+\sum_{i=\beta^{*}}^{p}\hat{\lambda}_i} {\alpha^{*}\kappa_{n}^2+p-\beta^{*}+1},
\end{eqnarray*}
and the optimal $\alpha^{*}$ and $\beta^{*}$ can be determined with about $O(p)$ operations. We note that $S_n$ is only required to be symmetric in \cite{TanaN14,LLV20}, and if we restrict it to be positive definite, the optimal $\hat{\Sigma}$ given by \eqref{eq.rk1} can also be derived from \cite{TanaN14,LLV20}. But, just as we have emphasized, in addition to characterizing the solution to \eqref{eq.covcd}, we consider the efficient construction of eigenvectors of $\hat{\Sigma}$, which was not discussed in \cite{TanaN14,LLV20}. In addition, Theorem~\ref{Thm1} and equation \eqref{eq.rk1} also show that when $n\geq p$ the condition number of $\hat{\Sigma}$, i.e. $\kappa(\hat{\Sigma})$, can be smaller than $\kappa_n$, whereas in the high-dimensional data setting $\kappa(\hat{\Sigma})$ is always equal to $\kappa_n$.
}
\end{remark}

\subsection{Algorithm}
\label{subsec.algthm}
{\color{black}
The construction of the optimal solution to the C3MA problem \eqref{eq.covcd} should be divided into two parts: searching for the optimal $\alpha^{*}$ and $\beta^{*}$, and constructing the eigenvectors. In the proof of Theorem~\ref{Thm1}, we have shown that the optimal $\alpha^{*}$ and $\beta^{*}$ can be found in $O(n)$ operations, which is faster than the binary search strategy given in \cite{TanaN14}. Here, we need to point out that the original idea for searching $\alpha^{*}$ and $\beta^{*}$ in $O(n)$ operations was proposed in the technical report \cite{Won06} for solving condition number constrained maximum likelihood covariance matrix estimation. In addition, the geometric perspective on search path was also given in \cite{Won06} and \cite{WonL13}. Based on the geometric idea given in \cite{Won06, WonL13}, we can design an algorithm to find the optimal $\alpha^{*}$ and $\beta^{*}$ in Theorem~\ref{Thm1} with $O(n)$ operations. The same problem was recently addressed by \cite{LLV20} and an algorithm was also presented there \cite[Algorithm~2.1]{LLV20}. A straightforward comparison will show that the Algorithm 2.1 in \cite{LLV20} is actually adapted from the Algorithm 1 in \cite{Won06} by changing the expression of $\mu$. The Algorithm 2.1 in \cite{LLV20} may be directly used to find the optimal $\alpha^{*}$ and $\beta^{*}$ for the C3MA problem \eqref{eq.covcd} by setting the lower bound equal to zero, and to avoid replicative work we refer the interested readers to \cite{LLV20} for implementation details.

}

Now we present how to utilise the special structure of $S_n$ in the high-dimensional data setting to compute the eigenvectors of $\hat{\Sigma}$ with matrix decomposition techniques.
According to \eqref{eq.defSn}, the sample covariance matrix $S_n$ is positive semidefinite in the high-dimensional data setting.
As we have assumed that $X_n$ is of full-column rank, we obtain $\rank(S_n)=n$ and the spectral decomposition of $S_n$ is given by
\begin{equation}
\label{eq.simEigSn}
  S_n=U\begin{bmatrix}
        \hat{\Lambda}_1 & 0 \\
        0  & 0\\
         \end{bmatrix} U^T,
\end{equation}
where $\hat{\Lambda}_1=\mathrm{Diag}(\hat{\lambda}_1, \hat{\lambda}_2, \cdots,\hat{\lambda}_n)$ with $\hat{\lambda}_1\geq\hat{\lambda}_2 \geq \cdots \geq \hat{\lambda}_n>0$, $U\in \mathbb{R}^{p\times p}$, and $UU^T=I_p$.
Since $\rank(X_n)=n$, the singular value decomposition (SVD) of ${1}/{\sqrt{n}}X_n$ can be written as follows
\begin{eqnarray*}
  \frac{1}{\sqrt{n}}X_n=U\hat{\Delta} V^T,
\end{eqnarray*}
where $V\in \mathbb{R}^{n\times n}$, $VV^T=I_n$, and $\hat{\Delta}=\begin{bmatrix}
                                                               \hat{\Delta}_1 \\
                                                               0 \\
                                                             \end{bmatrix}$ with
$\hat{\Delta}_1=\mathrm{Diag}(\hat{\delta}_1, \hat{\delta}_2, \cdots,\hat{\delta}_n)\in \mathbb{R}^{n\times n}$ and $\hat{\delta}_1\geq \hat{\delta}_2\geq \cdots \geq \hat{\delta}_n>0$  \cite[Ch.~8]{Golub13}. Then
\begin{eqnarray*}
  S_n=\frac{1}{{n}}X_nX_n^T=U\begin{bmatrix}
                           \hat{\Delta}_1 \\
                           0 \\
                         \end{bmatrix}
  V^TV\begin{bmatrix}
        \hat{\Delta}_1 & 0 \\
      \end{bmatrix}U^T=U\begin{bmatrix}
                          \hat{\Delta}_1^2 & 0 \\
                          0 & 0 \\
                        \end{bmatrix}
      U^T,
\end{eqnarray*}
which is exactly the spectral decomposition \eqref{eq.simEigSn} with $\hat{\Lambda}_1=\hat{\Delta}_1^2$.

We note that in \cite{TanaN14} the full spectral decomposition was suggested to find the optimal solution to \eqref{eq.covcd}. However, in the high-dimensional data setting the computation of full spectral decomposition of $S_n$ may be very time-consuming, which requires about $O(9p^3)$ operations when the symmetric QR algorithm is used. The preceding paragraph shows that if the left singular vectors $U$ and singular values $\hat{\Delta}_1$ of ${1}/{\sqrt{n}}X_n$ are available, the optimal solution to \eqref{eq.covcd} can also be derived. Besides, if the SVD is employed, its computational complexity is about $O(4p^2n - 8pn^2)$ when $U$ and $\hat{\Delta}$ are needed and computed by Golub-Reinsch SVD (GR-SVD) algorithm \cite{Golub13}. When $p$ is much larger than $n$, Chan showed that the computational efficiency of GR-SVD can be further improved by combining the QR factorization and GR-SVD algorithm, which was called the MOD-SVD \cite{Chan82}. Thus, with respect to computational complexity,  using SVD can lead to a significant reduction of computational burden compared with the full spectral decomposition in the high-dimensional data setting.

Therefore, we employ the MOD-SVD to compute the left singular vectors $U$ and singular values $\hat{\Delta}_1$ of ${1}/{\sqrt{n}}X_n$. Specifically, we first compute the QR factorization of ${1}/{\sqrt{n}}X_n$ with Householder transformation
\begin{eqnarray*}
  \frac{1}{\sqrt{n}}X_n = Q\begin{bmatrix}
                             R \\
                             0 \\
                           \end{bmatrix},
\end{eqnarray*}
where $Q\in \mathbb{R}^{p\times p}$, $QQ^T=I_p$ and $R\in \mathbb{R}^{n\times n}$ is an upper triangular matrix, and then the SVD of $R$ is computed by GR-SVD algorithm
\begin{eqnarray*}
  R = U_1\hat{\Delta}_1V^T,
\end{eqnarray*}
where $U_1\in \mathbb{R}^{n\times n}$ and $U_1U_1^T=I_n$. Finally, the SVD of ${1}/{\sqrt{n}}X_n$ is given by
\begin{eqnarray*}
  \frac{1}{\sqrt{n}}X_n = Q\begin{bmatrix}
                             U_1 & 0 \\
                             0 & I_{p-n} \\
                           \end{bmatrix}
\begin{bmatrix}
  \hat{\Delta}_1 \\
  0 \\
\end{bmatrix}V^T.
\end{eqnarray*}
A detailed comparison of computational complexity and storage requirement between GR-SVD and MOD-SVD can  be found in \cite{Chan82}. We would like to apply these two different algorithms to compute the SVD of the factor ${1}/{\sqrt{n}}X_n$ in the present paper, and a numerical comparison will be given in Example~\ref{exm1}.

Summarizing the above discussion, we present the following three algorithms for solving the C3MA problem \eqref{eq.covcd}. The first one is based on full spectral (FU-SPT) decomposition of $S_n$, which was suggested in \cite{TanaN14}. The other two are proposed algorithms based on the GR-SVD and MOD-SVD, respectively.
\begin{algorithm}[htbp]
{\small\caption{FU-SPT based solver for the C3MA problem }\label{Alg1}
{\textbf{Input:}  upper bound of the condition number $\kappa_n$, and data matrix $X_n$ .}\\
{\textbf{Output:} optimal approximation $\hat{\Sigma}$.}
\begin{enumerate}
  \item compute $S_n=\frac{1}{n}X_nX_n^T$.
  \item compute the full spectral decomposition $S_n$,
  \begin{equation*}
   S_n=U\hat{\Lambda} U^T,
  \end{equation*} and output $U$ and $\hat{\Lambda}$.
  \item find the optimal $u^{*}$ with $\hat{\Lambda}$, and give the optimal $\Lambda^{*}$ similar to Theorem~\ref{Thm1}.
  \item construct $\hat{\Sigma}$ by
\begin{equation*}
  \hat{\Sigma}=U\Lambda^{*}U^T.
\end{equation*}
\end{enumerate}}
\end{algorithm}

\begin{algorithm}[htbp]
{\small\caption{GR-SVD based solver for the C3MA problem }\label{Alg2}
{\textbf{Input:}  upper bound of the condition number $\kappa_n$, and data matrix $X_n$ .}\\
{\textbf{Output:} optimal approximation $\hat{\Sigma}$.}
\begin{enumerate}
  \item compute the SVD of ${1}/{\sqrt{n}}X_n$,
  \begin{equation*}
    \frac{1}{\sqrt{n}}X_n=U\begin{bmatrix}
                           \hat{\Delta}_1 \\
                           0 \\
                         \end{bmatrix}V^T,
  \end{equation*}
 and only output $U$ and $\hat{\Delta}_1$.
  \item find the optimal $u^{*}$ with $\hat{\Delta}_1^2$, and give the optimal $\Lambda^{*}$ by Theorem~\ref{Thm1}.
  \item construct $\hat{\Sigma}$ by
\begin{equation*}
  \hat{\Sigma}=U{\Lambda}^{*}U^T.
\end{equation*}
\end{enumerate}}
\end{algorithm}

\begin{algorithm}[htbp]
{\small \caption{MOD-SVD based solver for the C3MA problem }\label{Alg3}
{\textbf{Input:}  upper bound of the condition number $\kappa_n$, and data matrix $X_n$ .}\\
{\textbf{Output:} optimal approximation $\hat{\Sigma}$.}
\begin{enumerate}
  \item compute the QR factorization of ${1}/{\sqrt{n}}X_n$,
  \begin{equation*}
    \frac{1}{\sqrt{n}}X_n=Q\begin{bmatrix}
                           R \\
                           0 \\
                         \end{bmatrix},
  \end{equation*} and output $Q$ and $R$.
  \item compute the SVD of $R$,
  \begin{equation*}
    R=U_1\hat{\Delta}_1V^T,
  \end{equation*} and only output $U_1$ and $\hat{\Delta}_1$.
  \item find the optimal $u^{*}$ with $\hat{\Delta}_1^2$, and give the optimal $\Lambda^{*}$ by Theorem~\ref{Thm1}.
  \item construct $\hat{\Sigma}$ by
\begin{equation*}
  \hat{\Sigma}=Q\begin{bmatrix}
                  U_1 & 0 \\
                  0 & I_{p-n} \\
                 \end{bmatrix}{\Lambda}^{*}
\begin{bmatrix}
   U_1^T & 0 \\
   0 & I_{p-n} \\
 \end{bmatrix}Q^T.
\end{equation*}
\end{enumerate}}
\end{algorithm}

\begin{remark}
\label{rmk2}
{\color{black}The high-dimensional data setting ensures the sample covariance matrix $S_n$ has this special structure that allows us to use the SVD to construct the desired eigenvectors. But, if the factor structure is not available or $n$ is approximately equal to $p$, the proposed method will not be applicable or save more computational burden. Moreover, it is also the high-dimensional data setting that makes the MOD-SVD enjoy computational superiority over GR-SVD. Chan \cite{Chan82} showed that when $p\gtrsim 2n$ the MOD-SVD can achieve as much as 50 percent savings compared with the GR-SVD algorithm, and this has been incorporated into LINPACK \cite{EBai}. Based on the LINPACK routines, the MATLAB function \texttt{svd()} was built to compute the SVD of a matrix. However, as an interesting finding, we note that when the data matrix is large, the MOD-SVD should be introduced earlier, that is, when $p/n>1$ the MOD-SVD should also be used, not necessarily until $p/n\approx 2$. We will show this in our numerical Example~\ref{exm1}.}
\end{remark}

\subsection{{\color{black}Further discussion on $\kappa_n$}}
\label{subsec33}
Theorem~\ref{Thm1} shows that $\kappa_n$ is exactly the condition number of $\hat{\Sigma}$ and also determines the truncation positions of the eigenvalues of $S_n$ through the optimal $\mu^{*}$. In this part, we present some interesting discussion on $\kappa_n$, which may be of interest to be taken as future research.

\subsubsection{The influence of $\kappa_n$ on truncation process}
Define the univariate functions $\mu=\mu(\kappa_n)$ and $\nu=\nu(\kappa_n)=\kappa_n\mu$, then we can easily find that $\mu$ and $\nu$ are continuous functions on $R_{\alpha,\beta}$ given in the Appendix. Li et al. \cite{LLV20} also showed that the trajectory path for $\mu$ and $\nu$ is continuous on $\mu$-$\nu$-plane. $\mu$ and $\nu$ are used to cut the eigenvalues of $S_n$ from opposite sides. Intuitively, the distance between $\mu$ and $\nu$ should monotonically vary with $\kappa_n$, that is,  larger $\kappa_n$ should lead to a wider range between $\mu$ and $\nu$. This is true for condition number constrained maximum likelihood estimation of covariance matrix \cite{WonL13}.
Won et al. \cite{WonL13} showed that when the log-likelihood function is used to estimate the covariance matrix, it can be proved that $\mu$ is non-increasing in $\kappa_n$ and $\nu$ is non-decreasing, and both relationships hold almost surely. However, when the loss function is defined by \eqref{eq.covcd}, the monotone properties of $\mu$ and $\nu$, parallel to \cite[Proposition~1]{WonL13}, will not hold any more. In the following, we try to give some theoretical exploration to show that for the C3MA problem \eqref{eq.covcd} both $\mu$ and $\nu$ are not monotone functions of $\kappa_n$, and its numerical verification will be given in Example~\ref{exm2}.

For a given $\kappa_{n_0}$, let $\alpha$ be the largest index such that $\hat{\lambda}_{\alpha}>\nu(\kappa_{n_0})$, $\beta$ the smallest index such that $\mu(\kappa_{n_0})>\hat{\lambda}_{\beta}$, and the following inequalities hold
\begin{eqnarray*}
  \hat{\lambda}_{\alpha}>\nu(\kappa_{n_0}) &> \hat{\lambda}_{\alpha+1} \quad\mathrm{ and }\quad
  \hat{\lambda}_{\beta-1}> \mu(\kappa_{n_0})>\hat{\lambda}_{\beta}.
\end{eqnarray*}
Then we can find some $\kappa_n$ in a small neighbourhood of $\kappa_{n_0}$ such that
\begin{eqnarray*}
  \hat{\lambda}_{\alpha}>\nu(\kappa_n) &> \hat{\lambda}_{\alpha+1} \quad\mathrm{ and }\quad
  \hat{\lambda}_{\beta-1}> \mu(\kappa_n)>\hat{\lambda}_{\beta}.
\end{eqnarray*}
With basic calculus techniques, we can show that both
\begin{eqnarray*}
  \mu(\kappa_n) = \frac{\mathop\kappa_n\sum_{i=1}^{\alpha}\hat{\lambda}_i+\sum_{i=\beta}^{n}\hat{\lambda}_i} {\alpha\kappa_n^2+p-\beta+1} \; \; \mathrm{and} \;\;  \nu(\kappa_n) = \frac{\mathop\kappa_n^2\sum_{i=1}^{\alpha}\hat{\lambda}_i+\kappa_n\sum_{i=\beta}^{n}\hat{\lambda}_i} {\alpha\kappa_n^2+p-\beta+1}
\end{eqnarray*}
are not monotone functions of $\kappa_n\in[1, \infty)$, and its maximizers are given by
\begin{eqnarray*}
  \kappa_{\mu} = \max \left\{\sqrt{\left(\frac{\sum_{i=\beta}^{n}\hat{\lambda}_i}{\sum_{i=1}^{\alpha}\hat{\lambda}_i}\right)^2 + \frac{p-\beta+1}{\alpha}} - \frac{\sum_{i=\beta}^{n}\hat{\lambda}_i}{\sum_{i=1}^{\alpha}\hat{\lambda}_i}, 1\right\},
\end{eqnarray*}
and
\begin{eqnarray*}
  \kappa_{\nu} = \frac{p-\beta+1}{\alpha} \frac{\sum_{i=1}^{\alpha}\hat{\lambda}_i}{\sum_{i=\beta}^{n}\hat{\lambda}_i}+ \sqrt{\left(\frac{p-\beta+1}{\alpha} \frac{\sum_{i=1}^{\alpha}\hat{\lambda}_i}{\sum_{i=\beta}^{n}\hat{\lambda}_i} \right)^2 + \frac{p-\beta+1}{\alpha}},
\end{eqnarray*}
respectively. Therefore, we can not determine which is larger for $\mu(\kappa_{n_0})$ and $\mu(\kappa_n)$ or $\nu(\kappa_{n_0})$ and $\nu(\kappa_n)$, and the monotone properties parallel to \cite{WonL13} can not be established for the C3MA problem \eqref{eq.covcd}.

We can note that both $\mu$ and $\nu$ are unimodal functions.  If  $\gamma_1={\sum_{i=1}^{\alpha}\hat{\lambda}_i}/{\sum_{i=\beta}^{n}\hat{\lambda}_i}$ is large and also much larger than $\gamma_2={(p-\beta+1)}/{\alpha}$, which often empirically holds in the high-dimensional data setting according to our numerical experiments (see Example~\ref{exm2}),  then we have $\kappa_{\mu}=O(\sqrt{\gamma_2})$ and $\kappa_{\nu}=O(\gamma_1\gamma_2)$, and for $\kappa_n \geq \kappa_{n_0}$ in $[\kappa_{\mu}, \kappa_{\nu}]$ the following inequalities hold
\begin{eqnarray*}
  \mu(\kappa_{n_0}) \geq \mu(\kappa_n)\; \; \mathrm{ and }\; \; \nu(\kappa_{n_0}) \leq \nu(\kappa_n).
\end{eqnarray*}
The above inequalities show that the monotone properties corresponding to \cite[Proposition~1]{WonL13} may be established in a limited interval. But, strictly speaking, a theoretical justification of the assumptions on the truncated eigenvalues of $S_n$ is much more difficult. The discussion can only be taken as an illustration to show the possibility of establishing monotone properties of $\mu$ and $\nu$, and should not be taken as rigorous proof.

\subsubsection{The selection of $\kappa_n$}
The condition number of covariance matrix plays an important rule in studying the numerical stability of multivariate statistical models. For example, in binary classification with unknown but similar covariance structure, the discriminant vector $d$ is given by solving the positive definite linear system
\begin{eqnarray*}
  \hat{\Sigma} d =\bar{x}-\bar{y},
\end{eqnarray*}
where $\hat{\Sigma}$ is the estimated covariance structure, and $\bar{x}$ and $\bar{y}$ are the sample mean vectors of two different classes \cite{Ande03}. If we set $\Delta \Sigma$ be a small perturbation to $\hat{\Sigma}$ and its spectral norm $\|\Delta \Sigma\|_2<\epsilon$, then the standard perturbation result \cite[Ch.~2]{Golub13} shows that
\begin{eqnarray*}
  \frac{\|\hat{d}-d\|_2}{\|d\|_2}\leq \kappa(\hat{\Sigma}) \frac{\|\Delta \Sigma\|_2}{\| \hat{\Sigma}\|_2} +O(\epsilon^2),
\end{eqnarray*}
which means that when $\kappa(\hat{\Sigma})$ is large the relative error between the computed discriminant vector $\hat{d}$ and $d$ can also be very large. Another example is from portfolio management, which also stimulates this work. In investigating the perturbation theory of classical Markowitz mean-variance model \cite{Mkwz52}
\begin{eqnarray}\label{MKmodel}
  \min_{x\in\mathbb{R}^{p}}\frac{1}{2} x^T\Sigma x-u^Tx\quad \textrm{subject to }\mathbf{1}^Tx=1,
\end{eqnarray}
where $\Sigma$ is the covariance matrix of different investments, $u$ is the expected return vector, and $\mathbf{1}$ is a vector with its entries equal to $1$, we find that the rigorous or first order relative perturbation bounds of \eqref{MKmodel} can be unified into the following form
\begin{eqnarray*}
  \frac{\|\hat{x}-x\|_2}{\|x\|_2}\leq\kappa(\Sigma)f\left(\|\Delta\Sigma\|_2,\|\Delta u\|_2\right),
\end{eqnarray*}
where $\hat{x}$ is the computed solution, and $f$ is a function of perturbations $\Delta \Sigma$ and $\Delta u$ to $\Sigma$ and $u$, respectively \cite[Section~3]{WangY19}. Both of these two examples show that when the underlying true covariance matrix is unavailable, it should be important to take the condition number of covariance matrix into consideration in approximating the covariance matrix.

The above discussion also presents some inspiration in selecting $\kappa_n$ for the C3MA problem, that is, in order to get a small relative solution error, we may employ the perturbation results of applied models and give a suitable reduction to $\kappa_n$ in approximating the covariance matrix through \eqref{eq.covcd}. However, this user-chosen manner may seem to be a little artificial, and a data-driven method may give better reflection of its numerical essence and be more convincing. But $S_n$ is singular in the high-dimensional data setting and cannot be directly used to estimate the largest and smallest eigenvalues of $\Sigma$. According to Section~\ref{subsec.algthm}, a natural question is whether we can use the extreme singular values of ${1}/{\sqrt{n}}X_n$ to determine $\kappa_n$? The answer is still pessimistic, when $p$ is much larger than $n$. Let the elements of $X_n$ be independent standard normal random variables, then the expectation of the smallest singular value $\hat{\delta}_n$ of ${1}/{\sqrt{n}}X_n$ satisfies $\mathbf{E}\hat{\delta}_n\geq \sqrt{{p}/{n}}-1$ \cite{DaSz01}, which shows that it is the ratio $p/n$ that controls the smallest singular value of ${1}/{\sqrt{n}}X_n$, even for well-conditioned underlying covariance matrix. Some statistical methodologies were proposed to estimate the spectrum of random matrices, but these methods can not give reliable estimates of the extreme eigenvalues in the high-dimensional data setting \cite{ElK08,HenS09}. Thus some theoretical guaranteed data-driven procedures for selecting $\kappa_n$ in the high-dimensional data setting are still required and should be treated as future research work.

According to the condition number estimation theory, the accurate estimation of condition number is usually not required, and an estimate of the condition number within a factor 10 is usually acceptable \cite[Ch.~15]{High02a}. Thus for practical use we may consider the user-chosen $\kappa_n$ in \eqref{eq.covcd}. The idea for user-chosen parameter has been widely used in high-dimensional covariance matrix estimation. For example, to ensure the positive definiteness of covariance matrix, a common technique is to force its minimum eigenvalue larger than some given constant \cite{XueM12,LiuW14}. However, we found that bounding the minimum eigenvalue may lead to ill-conditioned covariance matrix estimation, when $p$ is much larger than $n$ \cite{WangY19}.  Thus considering numerical stability, it is more appropriate to bound the condition number, which guarantees both the positive definiteness and numerical stability of covariance matrix estimates. In practical applications, $\kappa_n$ may be chosen from the interval $[10^{4},10^6]$ , if the priori information on $\kappa_n$ is unavailable \cite{TanaN14}.

%
%Therefore, we consider the following data driven procedure to select $\kappa_n$ with sample covariance matrix $S_n$ instead of its factor $\frac{1}{\sqrt{n}}X_n$. Let $S^{K}_i=S_n(i+1:i+K,i+1:i+K)$ be the principal submatrix of $S_n$ of order $K$ with $i=0,\cdots,p-K$. Then by the interlacing property of symmetric matrix \cite{Golub13}, we have
%\begin{equation*}
%  \hat{\lambda}_j(S_n)\geq \lambda_j\left(S^{K}_i\right) \geq \hat{\lambda}_{p-K+j}(S_n),\; j=1,\cdots, K.
%\end{equation*}
%When $n>K$, a well known result is that $S^{K}_i$ is an asymptotically unbiased estimate of $\Sigma^{K}_i$, the corresponding principal submatrix of $\Sigma$ \cite{Ande03}. So we take the smallest eigenvalue of the principal submatrix $S^{K}_i$ as the

\section{Numerical experiment}
\label{sec.4}
In this part, {\color{black}we will present some numerical examples to show that the SVD-based algorithms enjoy much higher computational efficiency compared with the full spectral decomposition based method in the high-dimensional data setting. }Some numerical experiments are also given to illustrate the discussions given in Section~\ref{subsec33}. All the computations are performed in MATLAB R2014a on a PC with 4 GB RAM and Intel Core i5-6600 CPU running at 3.30 GHz.

\begin{example}
\label{exm1}
In this example, we will give a comparison of the three algorithms given in Section~\ref{subsec.algthm} with respect to different settings. The data is generated from multivariate normal distribution $\mathcal{N}(0,I_p)$ with sample size $n$, and we set $\kappa_n=10^3$ in all experiments for simplicity. For each pair of $p$ and $n$, we repeat the numerical experiment 100 times and report the mean values of the CPU time in seconds. We first give a comparison of Algorithms~\ref{Alg2} and \ref{Alg3} to verify the superiority of MOD-SVD. The  numerical results are reported in Table~\ref{Table1}, from which we can see that as the ratio $p/n$ increases the MOD-SVD based method becomes more and more efficient. This coincides with the conclusion given in \cite{Chan82}.

\begin{table}[htp]
  \centering
  {\small
  \caption{CPU time comparison of Algorithms~\ref{Alg2} and \ref{Alg3}. }\label{Table1}
  \begin{tabular}{|l|r|r|r|r|r|}
    \hline
    $n=100$& $p=150$ & $p=200$ & $p=250$ & $p=300$ & $p=350$ \\
    \hline
    GR-SVD  & 0.0316  & 0.0473 & 0.0826 & 0.1292 & 0.2061  \\
    MOD-SVD & 0.0215  & 0.0214 & 0.0224 & 0.0246 & 0.0259  \\
    \hline
  \end{tabular}}
\end{table}

In practical applications, the use of built-in function provided by data analysis software will substantially improve the performance of the proposed algorithm. To compare the three algorithms given in Section~\ref{subsec.algthm}, we employ the MATLAB functions \texttt{eig()}, \texttt{qr()}, and \texttt{svd()} in the following computations, and the numerical results are reported in Table~\ref{Table2}.  According to Table~\ref{Table2}, in the high-dimensional data setting the SVD based methods outperform the full spectral decomposition based method. However, when $p/n\gtrsim 2$, the first part of Table~\ref{Table2} shows that the GR-SVD based algorithm outperforms the MOD-SVD based algorithm that directly combines the function \texttt{qr()} and \texttt{svd()}. This is because the QR factorization step has been incorporated into the well programmed \texttt{svd()} function when $p/n\gtrsim 2$, which has been discussed in Remark~\ref{rmk2}.

\begin{table}[htp]
  \centering
  {\small
  \caption{CPU time comparison of the three algorithms. }\label{Table2}
  \begin{tabular}{|l|r|r|r|r|r|}
    \hline
    $n=500$ & $p=500$ & $p=1000$ & $p=2000$ & $p=3000$ & $p=4000$ \\
    \hline
    FU-SPT  & 0.0694 & 0.4319 & 4.9537 & 17.2565& 37.3999 \\
    GR-SVD  & 0.0692 & 0.1654 & 0.5851 & 1.5105 & 2.8391 \\
    MOD-SVD & 0.0710 & 0.1614 & 0.7262 & 2.0560 & 3.9470 \\
    \hline
    $n=2000$& $p=2500$ & $p=3000$ & $p=3500$ & $p=4000$ & $p=4500$ \\
    \hline
    FU-SPT  & 10.1859 & 16.6683 & 24.7184 & 37.5496 & 49.3076 \\
    GR-SVD  & 8.3104  & 10.4214 & 7.8683  & 9.4778  & 9.8960 \\
    MOD-SVD & 6.3723  & 7.0242  & 7.4115  & 9.2190  & 10.1207 \\
    \hline
  \end{tabular}}
\end{table}

As an interesting finding, we note that, when $n=2000$ and $p/n<2$, the MOD-SVD based algorithm outperforms the GR-SVD based one. For example, when $n=2000$ and $p=3000$ the MOD-SVD based algorithm can achieve even more than 30 percent savings compared with the GR-SVD based one. Therefore, we suggest that when the data set is large, even if $p$ is only slightly larger than $n$ the QR factorization step should also be incorporated into the \texttt{svd()} function.
\end{example}

\begin{example}
\label{exm2}
{\color {black}
In this example, we mainly consider the truncation process of $\kappa_n$, which should be treated as numerical complement to the discussions given in Section~\ref{subsec33}. The data is generated from the multivariate normal distribution $\mathcal{N}(0,\Sigma)$. The underlying true covariance matrix $\Sigma$ is constructed as follows
\begin{eqnarray*}
\Sigma=U\Lambda U^T,
\end{eqnarray*}
where $U$ is a random unitary matrix, and $\Lambda$ is a diagonal matrix with entries equally distributed from $10^i$ to $10^{-i}$. Thus the condition number of true covariance matrix is $\kappa(\Sigma)=10^{2i}$.

To investigate the truncation process of $\kappa_n$, we set $i=3$ equivalent to $\kappa(\Sigma)=10^6$, and increase $\kappa_n$ from $10$ to $10^4$ to check how the extreme eigenvalues of $S_n$ are truncated. We report the numerical results in Figure~\ref{Fig2}.  We note that when $\kappa_n=10$ the top panels of Figure~\ref{Fig2} show that both the larger and smaller eigenvalues of $S_n$ are truncated. However, as $\kappa_n$ increases, model \eqref{eq.covcd} tends to preserve the larger eigenvalues and the smaller ones are more likely to be truncated. When $\kappa_n=10^4$, only the smaller eigenvalues are truncated, this provides some support for why some high-dimensional covariance matrix estimation procedures would like to bound the eigenvalue from below \cite{XueM12,LiuW14}. For ease of identification of the truncation positions for $\kappa_n=10^2$ and $10^4$, the eigenvalues in the middle and bottom panels of Figure~\ref{Fig2} are given in logarithmic scale. Moreover, we find that as $\kappa_n$ increases the distance between two truncation positions becomes larger and larger.

\begin{figure}[htp]
  \centering
  \includegraphics[width=1\textwidth,height=0.8\textwidth]{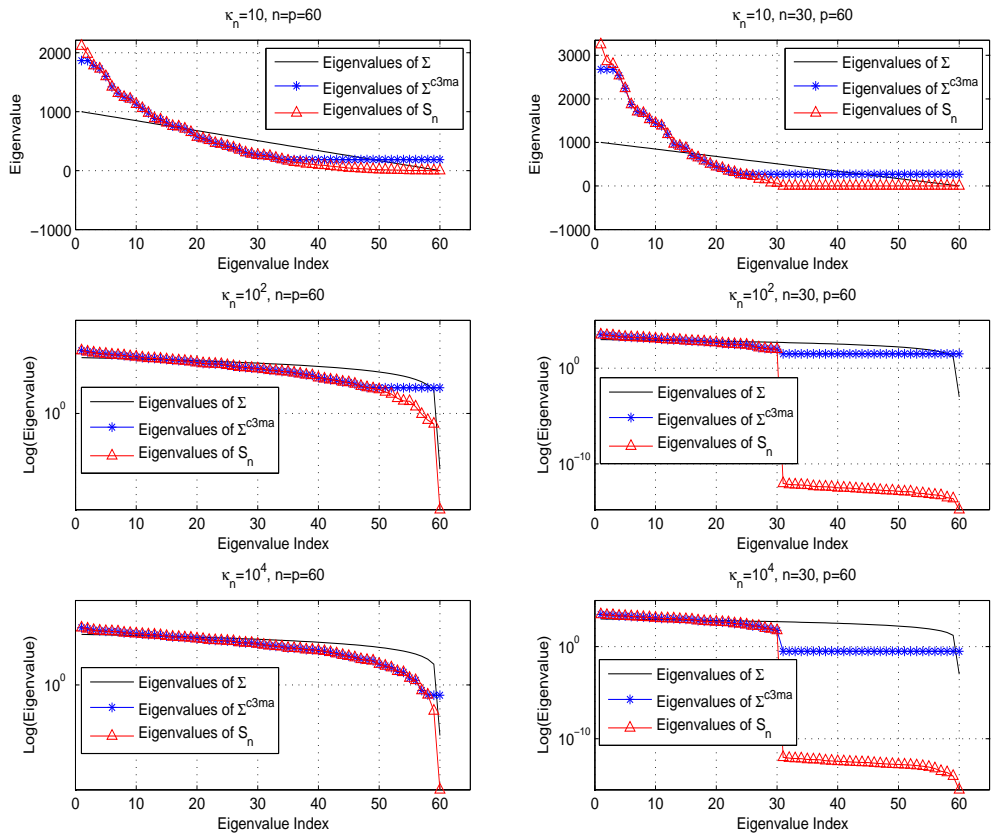}\\
  \caption{The truncation process of the C3MA model \eqref{eq.covcd} for $\kappa(\Sigma)=10^6$ with varying $\kappa_n$, $n$ and $p$. The black line is the true eigenvalue of $\Sigma$, the red triangles are the eigenvalues of $S_n$, and the blue stars are the eigenvalues of $\hat{\Sigma}$ given by \eqref{eq.covcd} and denoted by $\Sigma^{\mathrm{c3ma}}$.}\label{Fig2}
\end{figure}

In Section~\ref{subsec33}, we argue that both $\nu$ and $\mu$ are not monotone functions of $\kappa_n$, which is equivalent to say that as $\kappa_n$ increases both $\alpha$ and $\beta$ do not vary monotonically. For a better understanding, we design the following experiment.  Let $\kappa_n$ increase from 1 to some large value by a fixed step length, then for $\kappa_{n(i)}$ and $\kappa_{n(i+1)}$ we compute the corresponding successive differences $D_i(\beta)=\beta_{i+1}-\beta_{i}$ and $D_i(\alpha)=\alpha_{i+1}-\alpha_{i}$. To show the nonmonotone properties of $\nu$ and $\mu$, it is sufficient to check whether or not there exist negative values for $D_i(\beta)$ or positive values for $D_i(\alpha)$, respectively. In our experiment, we set $\kappa(\Sigma)=10^2$, increase $\kappa_n$ from 1 to 15 by step length 0.2,  plot the successive differences $D_i(\alpha)$ and $D_i(\beta)$ in Figure~\ref{Fig3}, and the negative successive difference for $\beta$ appears. This numerically verifies that the monotone property of $\mu$ or $\nu$ can not be established on the interval $[1,+\infty)$ with respect to $\kappa_n$.
\begin{figure}[htp]
  \centering
  \includegraphics[width=1\textwidth,height=0.5\textwidth]{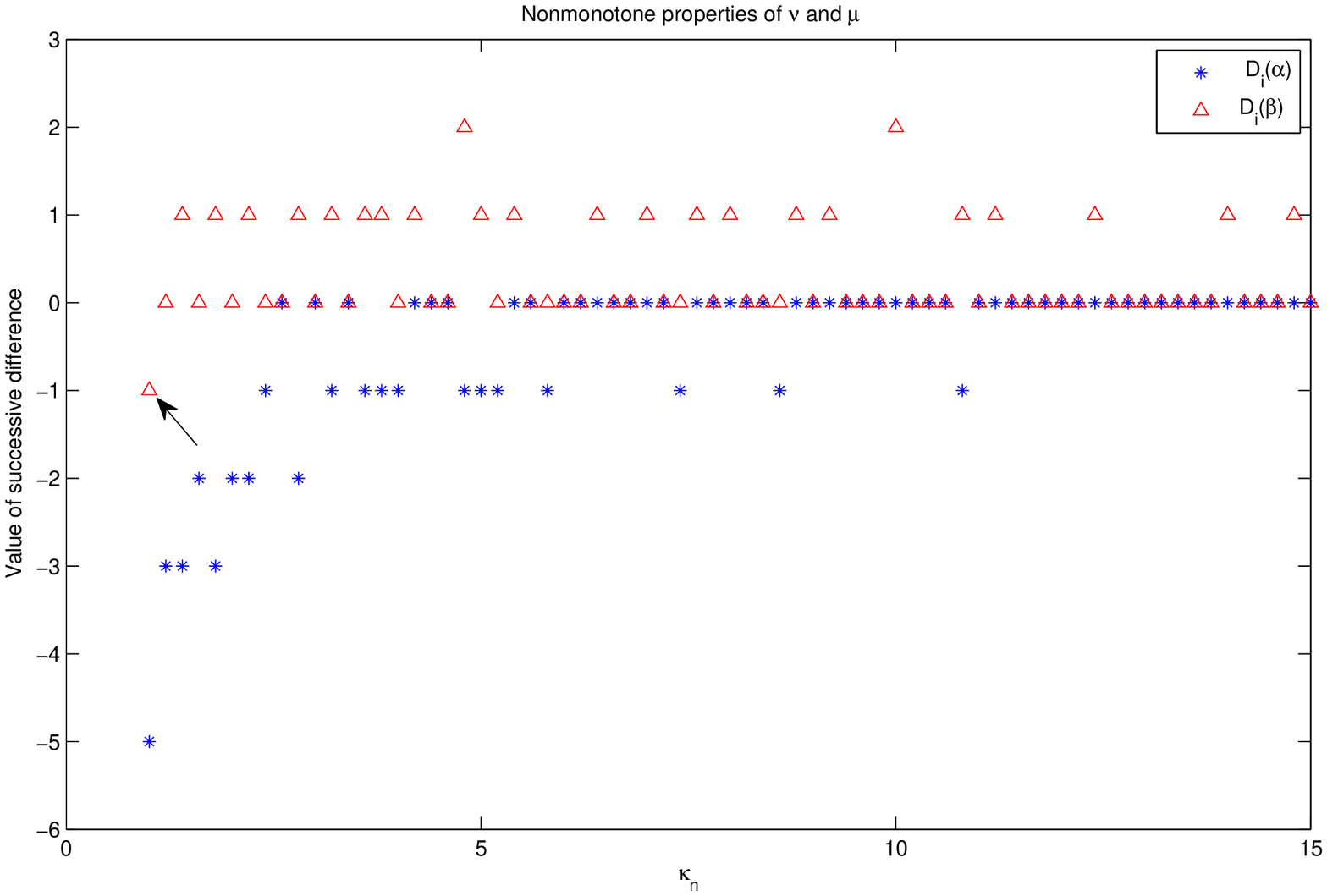}\\
  \caption{The nonmonotone properties of $\mu$ and $\nu$ with $\kappa(\Sigma)=10^2$ and $n=p=100$.}\label{Fig3}
\end{figure}

Despite monotone properties of truncation functions, we also want to check that under different regularization levels whether or not $\kappa_n$ is contained in the interval $[\kappa_{\mu}, \kappa_{\nu}]$. To show this, we set $\kappa(\Sigma)=10^4$ and repeat the numerical experiment 1000 times for each setting. For $\alpha$ and $\beta$, we report its smallest and largest values, and the percentages of $\kappa_n$ contained in $[\kappa_{\mu}, \kappa_{\nu}]$, denoted by IN, are also given.  The numerical results are reported in Table~\ref{Table3}. From Table~\ref{Table3}, we can easily note that although as $\kappa_n$ increases the average range between $\alpha$ and $\beta$ becomes wider, the smaller eigenvalues always tend to be truncated compared with the larger ones. This also coincides with Figure~\ref{Fig2}.
\begin{table}[htp]
  \centering
  \small
  \caption{The distribution of truncated eigenvalues}\label{Table3}
  \begin{tabular}{|l|l|c|c|c|c|}
    \hline
          &   & $\kappa_n=10$ & $\kappa_n=10^2$ & $\kappa_n=10^3$ & $\kappa_n=10^4$ \\
    \hline
  $n=p=400$& $\alpha$  & [10, 14]   & [1, 2]     & [1, 1]     & [1, 1]  \\
           & $\beta$   & [218, 224] & [322, 328] & [370, 375] & [389, 393]  \\
           & IN        & 100\%      & 100\%      & 100\%      & 100\% \\
    \hline
$n=200, p=400$ & $\alpha$  & [2, 6]     & [1, 1]     & [1, 1]     & [1, 1]  \\
               & $\beta$   & [150, 157] & [200, 200] & [200, 200] & [200, 200]  \\
               & IN        & 100\%      & 100\%      & 100\%     & 100\% \\
    \hline
  \end{tabular}
\end{table}
Moreover, in our simulation all $\kappa_n$s are contained in the interval $[\kappa_{\mu}, \kappa_{\nu}]$, and this may give some numerical support to the discussion given in Section~\ref{subsec33}. When $p$ is much larger than $n$ and the regularization level is not very high, the leading $n$ eigenvalues of $S_n$ are all reserved, which to some extent may explain that why we can not directly use the singular values of ${1}/{\sqrt{n}}X_n$ to determine $\kappa_n$ in the high-dimensional data setting.
}
\end{example}

\section{Concluding remark}
\label{sec.5}

In this paper, we gave a detailed investigation on the C3MA problem \eqref{eq.covcd} in the high-dimensional data setting and presented its explicit solution with respect to the Frobenius norm. By exploring the special structure of the high-dimensional data matrix, efficient SVD-based numerical algorithms were proposed to solve \eqref{eq.covcd}. Our numerical experiments showed that the proposed algorithms are quite efficient. It should be noted that it is the spiked rank structure of the data matrix that leads us to use SVDs to improve the performance of our algorithms. If the factor structure of the sample covariance matrix is not given, the proposed method in this paper will not be applicable. Note that the big data matrix are usually well approximated by low rank matrices \cite{UdeT19}. It would be of future interest to consider how the proposed method could be applied to symmetric and low rank data matrices $S_n$, where the factor structure is not available.

\section*{ Acknowledgement}
The author thank the editor, an associate editor and three anonymous reviewers for their helpful and detailed comments that generated a much better presentation of their work. Especially, the author also very appreciate one of the reviewers for his/her help in improving the proof of the main result.

\section*{Appendix}

\begin{proof}[Proof of Theorem~\ref{Thm1}]
Let $\lambda_1, \lambda_2, \cdots, \lambda_p$ be the eigenvalues of $\Sigma$  and satisfy $\lambda_1\geq \cdots\geq \lambda_p>0$. With Lemma~\ref{Lem.2}, problem \eqref{eq.lem2} amounts to
\begin{equation}
\label{eq.LM1}
 \Lambda^{*}= \mathop\mathrm{argmin}_{\exists \mu>0, 0<\mu\leq\lambda_p\leq\cdots\leq\lambda_1\leq\kappa_{n}\mu}\sum^{p}_{i=1}(\lambda_i-\hat{\lambda}_i)^2.
\end{equation}
If we define $f_i(\lambda_i,\hat{\lambda}_i)=(\lambda_i-\hat{\lambda}_i)^2$ with $i=1,\cdots, p$, then for a fixed $\mu$ the minimizer of $f_i(\lambda_i,\hat{\lambda}_i)$ is given by
\begin{eqnarray}
\label{eq.LM31}
 \lambda^{*}_i(\mu)&=& \mathop\mathrm{argmin}_{\exists \mu>0, 0<\mu\leq\lambda_i\leq\kappa_{n}\mu}f_i(\lambda_i,\hat{\lambda}_i)= \min\left\{\max(\mu,\hat{\lambda}_i),\kappa_{n}\mu\right\}
\end{eqnarray}
and
\begin{eqnarray*}
  f_i(\lambda_i^*(\mu),\hat{\lambda}_i)=\left\{
              \begin{array}{ll}
                (\kappa_{n}\mu-\hat{\lambda}_i)^2, & \hbox{$\mu< \frac{\hat{\lambda}_i}{\kappa_{n}}$;} \\
                0, & \hbox{$\frac{\hat{\lambda}_i}{\kappa_{n}} \leq \mu\leq \hat{\lambda}_i$;} \\
                (\mu-\hat{\lambda}_i)^2, & \hbox{$\mu>\hat{\lambda}_i$.}
              \end{array}
            \right.
\end{eqnarray*}
Clearly $f_i(\lambda_i^*(\mu),\hat{\lambda}_i)$ is convex and continuously differentiable.
In the high-dimensional data setting, $S_n$ has at most $n$ nonsingular eigenvalues and  $\hat{\lambda}_i=0$ when $i=n+1,\cdots, p$, and thus we have $\hat{\lambda}_1/\hat{\lambda}_p = \infty$. To meet the constraint $\kappa(\Sigma)\leq \kappa_n$, the last $p-n$ zero eigenvalues should be truncated and the optimal eigenvalues of $\Sigma$ are given by
\begin{eqnarray*}
  \mu=\lambda_p^*=\cdots = \lambda_{n+1}^*\leq\lambda_n^*(\mu)\leq\cdots\leq\lambda_1^*(\mu)\leq\kappa_{n}\mu.
\end{eqnarray*}

Therefore, finding the optimal $\Lambda^{*}$ in \eqref{eq.LM1} can be transformed into minimizing the following univariate function with respect to $\mu\in(0,+\infty)$
\begin{eqnarray}
\label{eq.LM2}
  f(\mu) &=&\sum_{i=1}^{p}f_i(\lambda_i^*(\mu),\hat{\lambda}_i)\nonumber\\
  &=&\sum_{i:\hat{\lambda}_i<\mu}f_i(\mu,\hat{\lambda}_i) +\sum_{i:\mu\leq\hat{\lambda}_i\leq\kappa_{n}\mu}f_i(\hat{\lambda}_i,\hat{\lambda}_i) +\sum_{i:\kappa_{n}\mu<\hat{\lambda}_i}f_i(\kappa_{n}\mu,\hat{\lambda}_i)\nonumber\\
  &=& \sum_{i:\hat{\lambda}_i<\mu}(\mu-\hat{\lambda}_i)^2 +\sum_{i:\kappa_{n}\mu<\hat{\lambda}_i}(\kappa_{n}\mu-\hat{\lambda}_i)^2.
\end{eqnarray}
Since $\hat{\lambda}_1\geq \cdots \geq \hat{\lambda}_n> \hat{\lambda}_{n+1}= \cdots =\hat{\lambda}_{p}=0$, it can be checked that $f(\mu)$ is strictly convex and continuously differentiable, and
\begin{eqnarray}
\label{eq.LM3}
  f^{'}(\mu) &=&\sum_{i:\hat{\lambda}_i<\mu}2(\mu-\hat{\lambda}_i) +\sum_{i:\kappa_{n}\mu<\hat{\lambda}_i}2\kappa_n(\kappa_{n}\mu-\hat{\lambda}_i)
\end{eqnarray}
and $f^{''}(\mu) =\sum_{i:\hat{\lambda}_i<\mu}2 +\sum_{i:\kappa_{n}\mu<\hat{\lambda}_i}2\kappa_n^2>0$. In order to determine the optimal $\mu^{*}$ that minimizes \eqref{eq.LM2}, we define
\begin{eqnarray*}
  R_{\alpha,\beta}&=&\{\mu:\hat{\lambda}_{\alpha} \geq \kappa_n\mu > \hat{\lambda}_{\alpha+1}\textrm{ and }\hat{\lambda}_{\beta-1}\geq\mu>\hat{\lambda}_{\beta}\}, \\
   f_{\alpha,\beta}(\mu)&=&\sum_{i=1}^{\alpha}f_i(\kappa_n\mu,\hat{\lambda}_i) +\sum_{i=\alpha+1}^{\beta-1}f_i(\hat{\lambda}_i,\hat{\lambda}_i) +\sum_{i=\beta}^{n}f_i(\mu,\hat{\lambda}_i) + (p-n)\mu^2\\
&=&\sum_{i=1}^{\alpha}(\kappa_n\mu-\hat{\lambda}_i)^2 +\sum_{i=\beta}^{n}(\mu-\hat{\lambda}_i)^2 + (p-n)\mu^2,
\end{eqnarray*}
for $\alpha\in\{1,\cdots,n\}$ and $\beta\in\{2,\cdots,n+1\}$ with $\beta-1\geq \alpha$.
Note that $f(\mu)=f_{\alpha,\beta}(\mu)$ for $\mu\in R_{\alpha,\beta}$, and we will also show that $\{R_{\alpha,\beta}\}$ can be used to form a partition of the region $(0,\hat{\lambda}_1]\times (0,\hat{\lambda}_1]$ in the next paragraph. These definitions provide basic tools for finding the optimal $\mu^{*}$ in $(0,\hat{\lambda}_1/\kappa_n]$. To show this, we can see that when $\mu>\hat{\lambda}_1/\kappa_n$ or equivalently $\kappa_n\mu>\hat{\lambda}_1$,  from \eqref{eq.LM3} we obtain
\begin{eqnarray*}
f^{'}(\mu)&=&\sum_{i:\hat{\lambda}_i<\mu}2(\mu-\hat{\lambda}_i) >0.
\end{eqnarray*}
Then, by the continuity of $f^{'}(\mu)$, when $\mu$ approaches $0$, we have
\begin{eqnarray*}
\lim_{\mu\rightarrow 0}f^{'}(\mu)=-\sum_{i:0<\hat{\lambda}_i}2\kappa_n\hat{\lambda}_i<0.
\end{eqnarray*}
Thus there exists a small neighbourhood of 0 in which $f^{'}(\mu)$ is strictly smaller than 0, and the unique minimizer $\mu^{*}$ of $f(\mu)$ must be contained in $(0,\hat{\lambda}_1/\kappa_n]$ and satisfies $f^{'}(\mu^*)=0$ due to the strict monotonicity of $f^{'}(\mu)$. In addition, with a similar analysis, we also note that for $f_{\alpha,\beta}(\mu)$ with $\mu\in(0,\infty)$ its unique minimizer is given by
\begin{eqnarray*}
  \mu_{\alpha,\beta} = \frac{\mathop\kappa_{n}\sum_{i=1}^{\alpha}\hat{\lambda}_i+\sum_{i=\beta}^{n}\hat{\lambda}_i} {\alpha\kappa_{n}^2+p-\beta+1},
\end{eqnarray*}
which may not be contained in $R_{\alpha,\beta}$ with the corresponding $\alpha$ and $\beta$. However, by the partitioning property of $\{R_{\alpha,\beta}\}$ and uniqueness of $\mu^{*}$ and $\mu_{\alpha,\beta}$, we can conclude that the optimal $\mu^{*}$ must be contained in some $R_{\alpha,\beta}$ and satisfies $\mu^{*}=\mu_{\alpha,\beta}$. This also provides the fundamental idea for searching the optimal $\mu^{*}$ in $O(n)$ operations.

To show the optimal $\mu^{*}$ can be found in $O(n)$ operations, we adapt the geometric idea given in \cite[Algorithm~1]{WonL13} and \cite{GuoZ17}. Let $\nu=\kappa_n \mu$, then $\left\{R_{\alpha,\beta}\right\}$ in $\mu-\nu$ plane partitions the region
\begin{eqnarray*}
  \left\{(\mu,\nu): 0<\mu\leq \hat{\lambda}_1, 0<\nu\leq \hat{\lambda}_1\right\}
\end{eqnarray*}
and the point $(\mu_{\alpha,\beta},\nu_{\alpha, \beta})$ is on the line $\nu=\kappa_n\mu$. So what we need to do is to search the optimal $\mu^*$ along the line $\nu=\kappa_n\mu$ and check whether $(\mu_{\alpha,\beta},\nu_{\alpha, \beta})$ is in $R_{\alpha,\beta}$. Because $\kappa_n \in [1,\infty)$ and $\hat{\lambda}_n>0$, if the line intersects $R_{\alpha,\beta}$, then the next intersection must occurs in $R_{\alpha+1,\beta}$, $R_{\alpha,\beta+1}$ or $R_{\alpha+1,\beta+1}$. Thus we need at most $2n$ tests to verify the condition $\mu_{\alpha,\beta}\in R_{\alpha,\beta}$. But different from \cite{WonL13,LLV20}, we do not need to search the first intersection since $\lambda_{n+1}=0$.
\end{proof}

\end{document}